\documentclass[reqno,12pt]{amsart}

\usepackage{enumerate}
\usepackage[margin=1in]{geometry}
\usepackage{ifpdf}
\usepackage{amsmath}
\usepackage{amsfonts}
\usepackage{amssymb}
\usepackage{amsthm}
\usepackage[ocgcolorlinks,hyperfootnotes=false,colorlinks=true,citecolor=blue,linkcolor=blue,urlcolor=blue]{hyperref}
\usepackage{setspace}
\usepackage{amsrefs}
\usepackage{nicefrac}
\usepackage{graphicx}
\usepackage{color}
\usepackage{mathtools}

\newcommand{\ignore}[1]{}

\renewcommand{\Re}{\operatorname{Re}}
\renewcommand{\Im}{\operatorname{Im}}

\newcommand{\abs}[1]{\left\lvert {#1} \right\rvert}
\newcommand{\sabs}[1]{\lvert {#1} \rvert}
\newcommand{\norm}[1]{\left\lVert {#1} \right\rVert}
\newcommand{\snorm}[1]{\lVert {#1} \rVert}

\newcommand{\C}{{\mathbb{C}}}
\newcommand{\R}{{\mathbb{R}}}

\newtheorem{thm}{Theorem}[section]

\newtheorem{prop}[thm]{Proposition}

\newtheorem{lemma}[thm]{Lemma}

\newtheorem{claim}[thm]{Claim}

\theoremstyle{definition}

\newtheorem{example}[thm]{Example}

\theoremstyle{remark}

\author{Ji\v{r}\'{\i} Lebl}
\thanks{The first author was in part supported by NSF grant DMS-1362337.}
\address{Department of Mathematics, Oklahoma State University,
Stillwater, OK 74078, USA}
\email{lebl@math.okstate.edu}

\author{Alan Noell}
\address{Department of Mathematics, Oklahoma State University,
Stillwater, OK 74078, USA}
\email{noell@math.okstate.edu}

\author{Sivaguru Ravisankar}
\address{School of Mathematics, Tata Institute of Fundamental Research,
Mumbai 400005, India}
\email{sivaguru@math.tifr.res.in}

\date{June 1, 2017}

\ifpdf
\hypersetup{
  pdftitle={On Lewy extension for smooth hypersurfaces in Cn x R},
  pdfauthor={Jiri Lebl, Alan Noell, and Sivaguru Ravisankar},
  pdfsubject={Several Complex Variables},
  pdfkeywords={Extension of CR functions, Lewy extension, CR singularity, Levi-flat},
}
\fi

\title{On Lewy extension for smooth hypersurfaces in $\C^n \times \R$}


\begin{document}


\begin{abstract}
We prove an analogue of the Lewy extension theorem for 
a real dimension $2n$ smooth submanifold
$M \subset \C^{n}\times \R$, $n \geq 2$.
A theorem of Hill and Taiani implies that
if $M$ is CR and the Levi-form has a positive
eigenvalue restricted
to the leaves of $\C^n \times \R$, then every smooth CR function $f$
extends
smoothly as a CR function to one side of $M$.  If the Levi-form
has eigenvalues of both signs,
then $f$ extends to a neighborhood of $M$.
Our main result concerns CR singular manifolds
with a
nondegenerate quadratic part $Q$.
A smooth CR $f$ extends to one side
if the Hermitian part of $Q$
has at least two positive eigenvalues, and $f$ extends to the other side if
the form has at least two negative eigenvalues.
We provide examples to show that at least two nonzero
eigenvalues in the direction of the extension are needed.
\end{abstract}

\maketitle




\section{Introduction} \label{section:intro}

Let $M \subset \C^n \times \R$ be a smooth real hypersurface.
A function $F \colon \C^n \times \R \to \C$ is a CR function if it is holomorphic
in the first $n$ variables.  A natural question is: When does a smooth
function $f \colon M \to \C$ extend
to a smooth CR function on a
neighborhood of $M$ in $\C^n \times \R$, or at least to one side of $M$.
The question is the classical CR extension, but with a real
parameter.

Let $(z,s) = (z_1,\dots,z_n,s)$ be the coordinates on $\C^n \times \R$.
Define the standard CR structure on $M$, the Cauchy-Riemann
equations restricted to $M$, as if it were embedded in
$\C^{n+1}$:  At each point $p \in M$ define
\begin{equation}
T_p^{(0,1)} M = \C \otimes T_p M
\,\cap\,
\operatorname{span}_{\C} \left\{
\frac{\partial}{\partial \bar{z}_1}
,\dots,
\frac{\partial}{\partial \bar{z}_n}
\right\} .
\end{equation}
Generically, we expect $\dim T_p^{(0,1)}M = n-1$, but
$\dim T_p^{(0,1)}M = n$ is possible.  If 
$\dim T_p^{(0,1)}M = n$ for all $p$,
then $M$ is a complex submanifold by the Newlander-Nirenberg theorem,
and so it is locally equal to $\C^n \times \{ s_0 \}$ for some $s_0$,
and $f$ extends to both sides of $M$ if and only if $f$ is holomorphic on $M$.
Thus, assume
that at least somewhere $\dim T_p^{(0,1)}M = n-1$.
The dimension $\dim T_p^{(0,1)}M$ is called the \emph{CR dimension} of $M$ at $p$.

The points where $\dim T_p^{(0,1)}M = n-1$ are called CR points of $M$,
and the points where $\dim T_p^{(0,1)}M = n$ are the so-called CR singularities.
Write $M_{CR}$ for the set of CR points of $M$.
We need a definition of a CR function on a possibly CR singular submanifold,
and we take the definition in the weakest possible sense:
\emph{A function $f \colon M \to \C$ is CR, if $v f = 0$ for all CR vector fields
on $M_{CR}$}.
Alternatively, we obtain the same definition if we allow $v$ to be smooth
vector fields on $M$ such that $v_p \in T^{(0,1)}_p M$ for all $p$. 
Such vector fields in general vanish at the CR singularities.
There are other possible
definitions of a CR function on a CR singular manifold,
but they imply the definition above.
The condition that $f$ be CR is clearly a necessary condition for $f$
to extend to even one side of $M$ in $\C^n \times \R$.
We prove that under certain nondegeneracy conditions
it is sufficient.

Notice $M \subset \C^n \times \R \subset \C^{n+1}$.
When $M$ and $f$ are real-analytic and both CR,
the function $f$
extends holomorphically to a neighborhood of $M$ in $\C^{n+1}$ as a
holomorphic function, via the classical theorem of Severi,
and hence real-analytically as a CR function of a
neighborhood of $M$ in $\C^n \times \R$.
If $M$ and $f$ are only
smooth and CR, such an extension need not hold.
For $M \subset \C^{n+1}$ not of infinite type, then $f$ would
extend into wedges, see for example Tumanov~\cite{Tumanov}.
But the $M$ we are considering is of infinite type, and $f$
does not in general extend into any open subset of $\C^{n+1}$.
The general Lewy extension
in the CR case was solved by Hill and Taiani~\cite{HillTaiani}: 
A CR function extends to a higher dimensional
submanifold extending in a direction of nonzero eigenvalues of the
Levi-form.

For every fixed $s$, and any set $X \subset \C^n \times \R$, define
\begin{equation}
(X)_{\{s\}} :=
 \{ z \in \C^n \mid (z,s) \in X \} .
\end{equation}
For a fixed $s$,
if $n \geq 2$, and the manifold $(M)_{\{s\}} \subset \C^{n} \times \{ s \}$ is not Levi-flat (the Levi-form has at
least one nonzero eigenvalue), then we obtain a local
extension of $f$ to at least one side of $(M)_{\{s\}}$ in
$\C^{n} \times \{ s \}$.  The key is to tie these extensions together.
In our setup, the result of Hill and Taiani~\cite{HillTaiani} 
implies the following corollary.  We state the theorem formally so
that the reader can compare the CR result with our CR singular
result below, Theorem~\ref{thm:extCRsing}, which is the
main result of the paper.

\begin{thm}[Special case of Hill-Taiani] \label{thm:extCR}
Let $M \subset \C^n \times \R$, $n \geq 2$, be a real smooth CR submanifold of
real dimension $2n$ (a hypersurface) and of CR dimension $n-1$.
Let $p = (z_0,s_0) \in M$ be a point.
Then there exists a neighborhood $U \subset \C^n \times \R$ of $p$,
such that given a smooth CR function $f \colon M \to \C$, we have:
\begin{enumerate}[(i)]
\item
If at $z_0$, $(M)_{\{s_0\}}$ is a hypersurface whose Levi-form has at least one positive
eigenvalue, then
the side of $M$ in $U$ corresponding to the positive
eigenvalue of the Levi-form of $(M)_{\{s_0\}}$
is a submanifold with boundary $H_+$, where
$\partial H_+ = M \cap U$,
and there exists a smooth function $F \colon H_+ \to \C$
that is CR in $H_+ \setminus M$ and
$f|_{M\cap U} = F|_{M\cap U}$.
\item
If at $z_0$, $(M)_{\{s_0\}}$ is a hypersurface whose Levi-form has
eigenvalues of both signs, then
there exists a
smooth CR function $F \colon U \to \C$, such that $f|_{M\cap U} =
F|_{M\cap U}$.
\end{enumerate}
\end{thm}

We provide a separate sketch of a proof in our notation
as some of the ideas are reused in
the CR singular case, and to make the paper more self contained.
The proof follows by modification
of the proof from \cite{crext}, where a similar, but global, theorem
was proved for $n \geq 1$; the case $n=1$ requires an extra hypothesis.

Let us move to CR singularities.
A real codimension two submanifold $M \subset \C^{n+1}$ has generically
isolated CR singularities.  Such manifolds were first studied in $\C^2$ by E.~Bishop~\cite{Bishop65}.
Especially in the elliptic case (when a family of discs attaches to $M$),
the work of Bishop was extended by Moser-Webster~\cite{MoserWebster83},
Moser~\cite{Moser85},
Kenig-Webster
\cites{KenigWebster:82}, Gong~\cite{Gong94:duke},
Huang-Krantz~\cite{HuangKrantz95}, 
Huang-Yin~\cite{HuangYin09}, Slapar~\cite{Slapar:16}, and many others.

For $n > 1$, the work mostly addressed normal form,
see Huang-Yin~\cites{HuangYin09:codim2,HuangYin:flattening1,HuangYin:flattening2},
Gong-Lebl~\cite{GongLebl}, Coffman~\cites{Coffman}, Burcea~\cites{Burcea,Burcea2}.
In particular, it is not always possible to change variables to realize the
manifold
as a submanifold of $\C^n \times \R$, a so-called \emph{flattening},
see Dolbeault-Tomassini-Zaitsev~\cites{DTZ,DTZ2},
Huang-Yin~\cites{DTZ,HuangYin:flattening1,HuangYin:flattening2}, and
Huang-Fang~\cites{HuangFang}.
If $M \subset \C^{n+1}$ is not \emph{flattenable},
then even in the real-analytic case, an extension of CR functions
does not in general exist near CR singularities. See Harris~\cite{Harris} and
Lebl-Minor-Shroff-Son-Zhang~\cite{LMSSZ}.

The authors considered the extension of CR functions for hypersurfaces
of $\C^n \times \R$
in the elliptic case~\cite{crext}, and in the general nondegenerate
real-analytic case~\cite{crext2}.  In this paper we address what 
happens in the nondegenerate smooth case.
Via the results of~\cite{crext2}, a formal extension
always exists, but a smooth extension does not exist in all cases.

Locally, after a complex affine change of variables
fixing $\C^n \times \R$, a smooth CR singular $M \subset \C^{n} \times \R$
is given by
\begin{equation} \label{eq:basicnorm}
\begin{split}
M: \quad s & = A(z,\bar{z}) + B(z,z)+\overline{B(z,z)} + E(z,\bar{z})
       = Q(z,\bar{z}) + E(z,\bar{z}) ,
\end{split}
\end{equation}
where $A$ is a Hermitian (sesquilinear) form, $B$ is a bilinear form,
$Q$ is the real quadratic form given by $A$ and $B$,
and $E$ is a smooth real-valued function in $O(3)$.
We require that
$Q$ is nondegenerate, that is,
$Q$ is represented by a nonsingular symmetric $2n \times 2n$ matrix.

Let
\begin{equation} \label{eq:quadmodel}
M^{quad}: \quad s = A(z,\bar{z}) + B(z,z)+\overline{B(z,z)} = Q(z,\bar{z})
\end{equation}
be the quadric model of $M$.
The manifold $M^{quad}$ (and therefore $M$)
has an isolated CR singularity because $Q$ is nondegenerate
(see Proposition~\ref{prop:isolgen}).
Suppose $A$ has $a$ positive and $b$ negative eigenvalues.
Diagonalize $A$ and write
\begin{equation} \label{eq:basicnormdiagA}
M: \quad s = \sum_{j=1}^a \sabs{z_j}^2 - \sum_{j=a+1}^{a+b} \sabs{z_j}^2 + B(z,z)
+ \overline{B(z,z)}  + E(z,\bar{z}),
\end{equation}
where $a=0$ or $b=0$ is interpreted appropriately.
Unless $A$ is positive definite we cannot always simultaneously diagonalize $B$.
If the number $a$ is normalized to $a \geq b$, then it is an invariant.
To simplify the statement of the theorem,
we do not make this requirement.

Define the manifold with boundary
\begin{equation} \label{eq:basicnormH}
H_+: \quad s \geq \sum_{j=1}^a \sabs{z_j}^2 - \sum_{j=a+1}^{a+b} \sabs{z_j}^2 + B(z,z)+\overline{B(z,z)} + E(z,\bar{z}) .
\end{equation}
The form $A$ carries the Levi form of the model manifold $(M^{quad})_{\{s\}}$, that is, $A$
restricted to $T_p^{(0,1)} (M^{quad})_{\{s\}}$ is the Levi-form of $(M^{quad})_{\{s\}}$.
Thus, a small perturbation of $A$ gives the
Levi-form of $(M)_{\{s\}}$.
If $a \geq 2$, then the Levi-form of $(M)_{\{s\}}$
has at least one positive eigenvalue in the direction of $H_+$.
Therefore
it is natural to expect an extension of $f$ to $H_+$, if $A$ has at least
two positive eigenvalues.
We now state our main result.

\begin{thm} \label{thm:extCRsing}
Suppose $M$ and $H_+$ are defined near the origin
by \eqref{eq:basicnormdiagA} and \eqref{eq:basicnormH}, $n \geq 2$,
and $Q$ is nondegenerate.
Then there exists a neighborhood $U$ of the origin, such that
given a smooth CR function $f \colon M \to \C$, we have:
\begin{enumerate}[(i)]
\item
If $A$ has at least two positive eigenvalues ($a \geq 2$),
then there exists
a function $F \in C^\infty(H_+ \cap U)$ such that
$F$ is CR on $(H_+ \setminus M) \cap U$ and $F|_{M \cap U} = f|_{M \cap U}$.
\item
If $A$ has at least two positive eigenvalues ($a \geq 2$)
and at least two negative eigenvalues ($b \geq 2$),
then there exists a CR function $F \in C^\infty(U)$ such that
$F|_{M \cap U} = f|_{M \cap U}$.
\end{enumerate}

In either case, $F$ has a formal power series in $z$ and $s$ at 0.
\end{thm}

By changing $s$ to $-s$, we change the roles of $a$ and $b$.
So if $A$ has two negative eigenvalues, then $f$ extends to
the opposite side of $M$ (below $M$).  In particular, the first item
immediately implies the second.

Some nondegeneracy is necessary.  As in
\cite{crext}, $M \subset \C^n \times \R$ given by $s = \snorm{z}^4$, and
$f \colon M \to \C$ given by $\snorm{z}^2 = \sqrt{s}$ is a
counterexample.  For this $M$, $(M)_{\{s\}}$ is strictly pseudoconvex
for $s > 0$, and
a continuous extension $F$ exists, but the extension is not differentiable
at the origin.

The conditions on the number of eigenvalues are necessary.
In 
\S\ref{section:prelimsexamples}, we show that if $M$ is given by
$s = \sabs{z_1}^2-\sabs{z_2}^2$, then there exists a smooth CR function that
extends to neither side.  In this case, both $a=1$ and $b=1$.  Furthermore,
if $M$ is given by
$s = \sabs{z_1}^2+\sabs{z_2}^2 - \sabs{z_3}^2$, then there exists
a smooth CR function that
extends only to $H_+$ and not the to opposite side.
In \S\ref{section:furtherexamples}, we also provide an example where $M$ is defined by
$s = \sabs{z_1}^2+\sabs{z_2}^2$, and there is a smooth CR
function on $M$ that extends to one side only near every point.

It is not clear if the isolated CR
singularity of $M$ and $M^{quad}$ is required
for the result, but our technique breaks down in several key steps
without this requirement, notably in the application of
Malgrange's theorem.  The results of \cite{crext2} show
that a formal extension exists even when the CR
singularity is not isolated (as long as $A$ is nondegenerate).

We mostly discuss $n \geq 2$.  When $n=1$, the CR condition is
vacuous, and every smooth function is a CR function.
An extra condition is necessary for extension, and under an extra
condition the authors proved an extension in the elliptic case
in~\cite{crext}.

A natural question is whether a version of the well-known Baouendi-Tr\`eves approximation theorem holds in the present setting.  Example~\ref{example:oneone} shows that there is no straightforward analogue.

Let us outline the organization of the paper.  We prove some
preliminaries and give several examples
in \S\ref{section:prelimsexamples} to justify the hypotheses of the theorems.
In \S\ref{section:extCR} we sketch
the proof of the extension at CR points.
We construct certain affine analytic discs in \S\ref{section:affinedisks} that are useful to study the extension to $H_+\setminus M$.
\S\ref{section:extnnearsing} is devoted to showing the existence and regularity of the extension in $H_+\setminus M$ near the CR singularity.
The affine discs from \S\ref{section:affinedisks} are used to show apriori interior regularity of the extension and to show the existence of the extension for $s\le 0$.
Additionally we also use natural complex manifolds with boundary attached to $M$ to show existence of the extension for $s > 0$.
In \S\ref{section:formal} we use the results of \cite{crext2} to show that a
formal extension exists at CR singularities.
We prove that the extension is regular up to the boundary
in \S\ref{section:regularityCRsing}.
Finally, in \S\ref{section:furtherexamples} we provide further relevant examples.


\section{Preliminaries and examples} \label{section:prelimsexamples}

The sesquilinear quadratic form $A(z,\bar{z})$ is represented by
a Hermitian matrix, which we also
call $A$, as 
$\bar{z}^t A z$, thinking of $z$ as a column vector.
Similarly, the bilinear form  $B(z,z)$ is represented by a matrix $B$ as 
$z^t B z$.  The matrix $B$ is not
unique, but it can be made unique by requiring that $B$
is symmetric, or
upper (or lower) triangular.  Let us suppose that it is symmetric.

A linear transformation $T$ in the $z$ variables transforms
the matrices $A$ and $B$
as $T^*AT$ (where $*$ is the conjugate transpose) and $T^tBT$.
Via Sylvester's law of inertia, we change $A$
into a diagonal matrix with only $1$s, $0$s, and $-1$s
on the diagonal.
Both matrices cannot in general be diagonalized
by the same $T$.  However, it is a standard result in linear
algebra, that if $A$ is positive
definite, then the symmetric $B$ can be diagonalized,
such that the diagonal elements
are nonnegative numbers.  See e.g.\ \cite{HornJohnson}*{Theorem 7.6.5}.
In other words, if $A$ is positive definite, then $M^{quad}$
can be put into the form
\begin{equation} \label{eq:hermposQ}
M^{quad}: \quad s = \snorm{z}^2 + \sum_{j=1}^n \lambda_j (z_j^2+\bar{z}_j^2) =
\sum_{j=1}^n \left( \sabs{z_j}^2 + \lambda_j
(z_j^2+\bar{z}_j^2) \right) .
\end{equation}

We start by showing that having an isolated CR singularity of $M^{quad}$ is
a generic condition.  When $A$ is positive definite, it is
not difficult to see this by direct computation.
In this case, $M^{quad}$ has an isolated CR singularity if and only if
$\lambda_j \not= \frac{1}{2}$ for all $j$, what is normally called
``nonparabolic''.  The general case is similar.

\begin{prop} \label{prop:isolgen}
Let a quadric model $M^{quad}$ be given by
\begin{equation}
s = \sum_{j=1}^a \sabs{z_j}^2 - \sum_{j=a+1}^{a+b} \sabs{z_j}^2 + B(z,z)
+ \overline{B(z,z)} = Q(z,\bar{z})
\end{equation}
for a bilinear form $B$.
$M^{quad}$ has an isolated CR singularity if and only if the real quadratic form $Q(z,\bar{z})$
is nondegenerate.
In particular, the set of $B$ for which $M^{quad}$
has an isolated CR singularity is an open dense set.
It is the
complement of a proper
real-algebraic subvariety of the set of all symmetric $n \times n$
matrices giving $B(z,z)$.

If $M$ is a submanifold with quadric model $M^{quad}$ such that $Q$
is nondegenerate, then $M$ also has an isolated CR singularity
at the origin.
\end{prop}

\begin{proof}
Suppose $s = Q(z,\bar{z})$ is the defining equation for $M^{quad}$.
The real quadratic form $Q$ is given by a
real $2n \times 2n$ symmetric matrix.
The set of CR singularities is the set
where the $z$-plane is tangent to $M^{quad}$, therefore, where
the derivative of $Q$ vanishes.  The derivative of $Q$ vanishes outside the origin
if and
only if the underlying real matrix is not of full rank.  The conclusion about
$B$ follows at once.  The final conclusion about $M$ also follows at once;
the $z$-plane $M$ cannot be tangent to $M$
unless the all the $z$ derivatives of $Q(z,\bar{z}) + E(z,\bar{z})$ vanish,
and this is not possible in some punctured neighborhood of the
origin if $Q$ is nondegenerate.
\end{proof}

\begin{example}
It is possible for a degenerate $M$ to have an isolated CR singularity.
The introduction provides the example $s = \snorm{z}^4$, which is degenerate
in the sense that both $A$ and $Q$ are degenerate,
but the CR singularity is the origin alone.

On the other hand, if $A$ is nondegenerate, it is still possible for $Q$ to be degenerate, and
for the CR singularity to be large.  For example for the ``completely
parabolic'' quadric manifold (all $\lambda_j = \frac{1}{2}$) the CR singularity is
a totally real $n$-real-dimensional linear submanifold.
\end{example}

Let us show via examples that the conditions of
Theorems~\ref{thm:extCR} and \ref{thm:extCRsing} are 
necessary.

\begin{example}
\emph{In the CR case, at least one nonzero
eigenvalue is necessary.}

Let $(z,s) \in \C \times \R$, and $M$ be given by $\Im z = 0$.
Define a smooth CR function $f \colon M \to \C$ by
\begin{equation}
f(z,s) =
\begin{cases}
\frac{e^{-1/s^2}}{z+is} & \text{if $s \not= 0$}, \\
0 & \text{if $s = 0$} .
\end{cases}
\end{equation}
Any extension must satisfy the same formula near $M$,
simply plugging in complex $z$,
being the unique extension along $(M)_{\{s\}}$.
It is clear that $f$ extends to neither side of $M$ in $\C \times \R$.
In this example $n=1$, where clearly Hill-Taiani does not apply because
there is no Levi-form.  It is immediate that when $n > 1$
and the Levi-form is zero (Levi-flat), the extension does not occur.

Similarly, let $M \subset \C^2 \times \R$ and $f$ be given by
\begin{equation}
\Im z_1 = s \abs{z_2}^2 ,
\qquad
f(z,s) =
\begin{cases}
\frac{e^{-1/s^2}}{z_1+is} & \text{if $s \not= 0$}, \\
0 & \text{if $s = 0$} .
\end{cases}
\end{equation}
We obtain
an example where $M$ is not Levi-flat,
the Levi-form is only zero when $s=0$, and the extension
to neither side is possible near the origin.
\end{example}

\begin{example}\label{example:oneone}
\emph{In the CR singular case, at least two eigenvalues of the
same sign are necessary to guarantee an extension to some side.}

Let $M \subset \C^2 \times \R$ be given by
\begin{equation}
s = \abs{z_1}^2 - \abs{z_2}^2 .
\end{equation}
Let $f \colon M \to \C$ be given by
\begin{equation}
f(z,s) =
\begin{cases}
\frac{1}{z_1} e^{-1/s^2} & \text{if $s > 0$,} \\
0 & \text{if $s = 0$,} \\
\frac{1}{z_2} e^{-1/s^2} & \text{if $s < 0$.}
\end{cases}
\end{equation}
Let us show that $f$ is smooth.  If $s > 0$, then clearly
$z_1$ is not zero on $M$.  Similarly if $s < 0$, then $z_2$
is not zero on $M$.  Therefore, $f$ is smooth when $s \not=0$.
By symmetry it is enough to show that $f$ is smooth if $s \geq 0$ as we
approach $s = 0$ from above.
Suppose $s > 0$.  Rewrite $f$ in terms of $z$ and $\bar{z}$.
\begin{equation}
f(z,\abs{z_1}^2-\abs{z_2}^2)
=
\frac{1}{z_1} e^{-1/{(\abs{z_1}^2-\abs{z_2}^2)}^2}
\end{equation}
The derivatives of $f$ are going to be of the form
\begin{equation}
\frac{P(z,\bar{z})}{z_1^d{(\abs{z_1}^2-\abs{z_2}^2)}^k}
\,
e^{-1/{(\abs{z_1}^2-\abs{z_2}^2)}^2}
\end{equation}
for a polynomial $P$.
As $s > 0$, and so
$\abs{z_1}^2 > \abs{z_2}^2$, we bound
\begin{equation}
\abs{\frac{1}{z_1}}
\leq
\frac{1}{\abs{z_1}^2-\abs{z_2}^2} .
\end{equation}
So all derivatives are  bounded on the set where $s > 0$, as $s$ approaches
0.  Therefore $f$ is smooth.
$f$ is also CR as we wrote it as a function of $z$
and $s$ and not $\bar{z}$.

On the other hand $f$ cannot extend holomorphically into the set where
$s > \abs{z_1}^2-\abs{z_2}^2$, because there is a pole arbitrarily
close to the origin (where $z_1 = 0$).
Similarly $f$ cannot extend holomorphically into the set
$s < \abs{z_1}^2-\abs{z_2}^2$.  Hence there is no neighborhood of the
origin and no CR function (of any regularity) on either side of $M$ in
$\C^2 \times \R$ inside this neighborhood that extends $f$.

This example shows that a straightforward analogue of the Baouendi-Tr\`eves approximation theorem does not hold.  In particular, the above CR function $f$ cannot be a uniform limit of polynomials in $z$ and $s$ in some compact neighborhood $K \subset \subset M$ of the origin.  Every point in $\C^2 \times \R$ lies on an analytic disc attached to $M$; to see this, simply fix $z_2$ and $s$, and the intersection of the resulting line with $M$ will give such a disc.  Therefore, if $f$ were such a limit, then it would extend to a neighborhood of the origin in $\C^2 \times \R$.
\end{example}

\begin{example}
\emph{In the CR singular case, we need at least two eigenvalues of both signs
to guarantee extension to both sides.}

Let $M \subset \C^3 \times \R$ be given by
\begin{equation}
s = \abs{z_1}^2 + \abs{z_2}^2 - \abs{z_3}^2 .
\end{equation}
Let $f \colon M \to \C$ be given by
\begin{equation}
f(z,s) =
\begin{cases}
0 & \text{if $s \geq 0$,} \\
\frac{1}{z_3} e^{-1/s^2} & \text{if $s < 0$.}
\end{cases}
\end{equation}
Let us show that $f$ is smooth.  If $s < 0$, then clearly
$z_3$ is not zero on $M$.  Therefore, $f$ is smooth when $s \not=0$.
The computation is similar as in the previous example.
Suppose $s < 0$.  Rewrite $f$ in terms of $z$ and $\bar{z}$.
\begin{equation}
f(z,\abs{z_1}^2+\abs{z_2}^2-\abs{z_3}^2)
=
\frac{1}{z_3} e^{-1/{(\abs{z_1}^2+\abs{z_2}^2-\abs{z_3}^2)}^2}
\end{equation}
The derivatives of $f$ are going to be of the form
\begin{equation}
\frac{P(z,\bar{z})}{z_3^d{(\abs{z_1}^2+\abs{z_2}^2-\abs{z_3}^2)}^k}
\,
e^{-1/{(\abs{z_1}^2+\abs{z_2}^2-\abs{z_2}^2)}^2}
\end{equation}
for a polynomial $P$.
As $s < 0$, then
$\abs{z_1}^2+\abs{z_2}^2 > \abs{z_3}^2$.  We bound
\begin{equation}
\abs{\frac{1}{z_3}}
\leq
\frac{1}{\abs{z_3}^2-(\abs{z_1}^2+\abs{z_2}^2)} .
\end{equation}
So all derivatives are  bounded on the set where $s < 0$, as $s$ approaches
0.  Therefore $f$ is smooth.
Again $f$ is CR as we wrote it as a function of $z$
and $s$, and not $\bar{z}$.

Similarly as for the previous example, the function cannot extend to the
set $s < \abs{z_1}^2+\abs{z_2}^2-\abs{z_3}^2$, because there is a pole arbitrarily
close to the origin (where $z_3 = 0$).  By Theorem~\ref{thm:extCRsing}, it does extend
to the other side smoothly.
\end{example}

\begin{example}
\emph{Interior regularity may fail for a degenerate submanifold with a large CR singularity.}

Write $z = (z',z'')$, and define $M$ by
\begin{equation}
s={\bigl(\norm{z'}^2-\norm{z''}^2\bigr)}^3 .
\end{equation}
Let $f \colon M \to \C$ be given by $f(z,s) = \sqrt[3]{s}$.  This function is CR, smooth on $M$, in fact real-analytic, and extends as a continuous CR function to a neighborhood of the origin.  But the extension is not differentiable when $s=0$.
\end{example}


\section{Extension at CR points} \label{section:extCR}

In this section we provide a sketch of the proof of the first part of
Theorem \ref{thm:extCR}, from which the second part immediately follows.
Suppose $p=(z_0,s_0)\in M$ is such that the Levi-form of the hypersurface
$(M)_{\{s_0\}}\subset\C^n$ has at least one positive eigenvalue at $p$.
There is a neighborhood $U$ of $p$, where the Levi-form of $(M)_{\{s\}}$ has
at least one positive eigenvalue at $(z,s)\in U$ corresponding to the same
side $H_+$.

Let $I = \{s\in \R \mid (z,s)\in H_+\cap U \text{ for some } z\in \C^n\}$.
$U$ can be made small enough so that
for each $s\in I$, every point of $\left(U \cap H_+\setminus M\right)_{\{s\}}$ is
contained in an analytic disk whose boundary lies on $M$.
Using the Cauchy integral formula on these analytic disks, a smooth CR function
$f\colon M\cap U\to\C$ extends to a function
$F\colon H_+\cap U\to\C$ that is holomorphic and smooth up to
the boundary on each leaf $(H_+)_{\{s\}}$.
This is Lewy's extension theorem.
Furthermore, the analytic disks above can be chosen to vary smoothly 
with respect to $s$, showing that
$F\in C^{\infty}\bigl((H_+\setminus M)\cap U\bigr)$.

It only remains to be shown that $F$ is smooth up to $M$ in $H_+\cap U$.
We do this by an approach that is similar to but simpler than that
employed in Lemma 4.2 from \cite{crext}.
We present an outline below that incorporates the simplifications and
the necessary modifications.

\begin{claim} $F$ is continuous $H_+\cap U$.
\end{claim}
\begin{proof}
It suffices to check for continuity up to $M$ in $H_+\cap U$.
For $(z_0,s_0)\in M\cap U$ and $(z,s)\in \left(H_+\setminus M\right)\cap U$,
\begin{equation}
\abs{F(z_0,s_0)-F(z,s)} \le \abs{F(z_0,s_0)-F(z,s_0)} + \abs{F(z,s_0)-F(z,s)}.
\end{equation}
Since $F\in C^{\infty}\left(\left(H_+\setminus M\right)\cap U\right)$,
it is enough to show that the first term can be made small.
For $z$ close to $z_0$ there is an analytic disk $\Delta$ passing through
$z$ and attached to $M\cap V$ for a neighborhood $V\subset U$ of $(z_0,s_0)$.
Considering the smooth CR function $f-f(z_0,s_0)$
and using the maximum modulus principle, we have
\begin{equation}
\abs{F(z,s_0)-F(z_0,s_0)}\le \sup\limits_{(\zeta,s_0)\in\partial\Delta} \abs{f(\zeta,s_0)-f(z_0,s_0)}.
\end{equation}
The conclusion follows since $f$ is smooth on $M$.
\end{proof}

\begin{claim} $F_s$ extends continuously to $H_+ \cap U$ and
the restrictions $F_{z_k}|_{M\cap U}$ and $F_s|_{M\cap U}$ are smooth.
\end{claim}
\begin{proof}
Let $(z_0,s_0)\in M\cap U$.
Without loss of generality we may suppose that $\rho_{z_1}\ne 0$ in a neighborhood $V\subset U$ of $(z_0,s_0)$. 
Since $M\cap V$ is CR, there are smooth vector fields $X_1,\dots,X_{n-1}$ that span $T^{(1,0)}(M\cap V)$.
The only missing tangential direction, also called the bad tangent direction, is given by
\begin{equation}
X=\frac{\partial}{\partial z_1} - \left(\frac{\rho_{z_1}}{\rho_{\bar{z}_1}}\right)\frac{\partial}{\partial \bar{z}_1}.
\end{equation}
Since $F$ is holomorphic along each leaf, $F_{z_1}\vert_{M\cap V} = XF\vert_{M\cap V}=Xf$ is smooth.
Also, $X_k F\vert_{M\cap V}=X_k f$, for $1\le k\le n-1$, are smooth as well.
So, $F_{z_1},\dots,F_{z_n}$ are smooth on $M\cap V$ since
$\{\partial/\partial z_1,X,X_1,\dots,X_{n-1},\bar{X}_1,\dots,\bar{X}_{n-1}\}$
is a smooth coordinate frame on each leaf in $V$.

To see that $F_s$ is smooth on $M$, consider the vector field
\begin{equation}
Y=\frac{\partial}{\partial s} - \left(\frac{\rho_s}{\rho_{\bar{z}_1}}\right)\frac{\partial}{\partial \bar{z}_1}
\end{equation}
on $V$.
Let $X$ be any smooth CR vector field on $M\cap V$.
Then, $Xf\equiv 0$ because $f$ is a CR function.
Since $X$ does not have a $\frac{\partial}{\partial s}$ component, neither does $[X,Y]$.
That is, $[X,Y]$ differentiates along the leaves and contains only $\bar{z}$ derivatives.
Consequently, $X(Yf)=[X,Y]f=0$ since $F$ is holomorphic and smooth up to $M$ in each leaf.
So, $Yf$ is a smooth CR function on $M\cap V$, and hence it extends to a continuous
CR function $G$ on $H_+\cap U$ that is smooth inside, holomorphic on each leaf,
and smooth up to $M$ on each leaf.
Since $F_s=G$ on $\left(H_+\setminus M\right)\cap V$, $F_s$ extends continuously
to $M$ and is smooth on $M$.
\end{proof}

The earlier claim shows that $F_{z_k}$ and $F_s$ are continuous
on $H_+\cap U$ since they are smooth CR functions on $M\cap U$.
That is, $F\in C^1(H_+\cap U)$.
Repeating this procedure, we find that $F$ is smooth on $H_+\cap U$.


\section{Affine analytic discs} \label{section:affinedisks}

First we need a small proposition that is essentially a version of Thom's transversality theorem for
affine discs and real hypersurfaces.

\begin{prop} \label{prop:transversal}
Suppose $\Omega \subset \C^n$, $n \geq 2$, is a domain
with smooth boundary and
$L_\alpha(\xi) = \alpha \xi + \beta$ an affine map where $\alpha \in\C^n$,
$\beta \in \Omega$, and $\xi\in\C$.
Denote by $D_{\alpha}$ the connected component of $L_{\alpha}(\C) \cap
\Omega$ containing $\beta$.
Suppose $D_{\alpha}$ is bounded.

Then given any $\epsilon > 0$ and any neighborhood $U$ of the closure
$\overline{D_{\alpha}}$,
there exists an $\tilde{\alpha } \in \C^n$ with $\norm{\alpha -\tilde{\alpha }} < \epsilon$,
and $\alpha_k = \tilde{\alpha }_k$ for $k=3,\dots,n$,
such that $\overline{D_{\tilde{\alpha }}} \subset \subset U$,
and $L_{\tilde{\alpha }}(\C)$ intersects $\partial \Omega$ transversally.
\end{prop}

\begin{proof}
Without loss of generality, by restricting to a 2 dimensional subspace we
assume $n=2$ for simplicity, and we also assume that $\alpha \in S^{3} \subset
\C^2$.  Write $L(\xi,\alpha) = L_\alpha(\xi) = \alpha \xi + \beta$ and treat $L \colon \C
\times S^3 \to \C^2$ as a mapping of manifolds.  Clearly, $L(\C \times S^3)
= \C^2$, and hence $L$ is transverse to $\partial \Omega \subset \C^2$.  By
Thom's transversality theorem then $L_{\alpha}$ is transverse to $\partial \Omega$
for almost all $\alpha$.  The compactness follows at once for $\tilde{\alpha}$ close
enough to $\alpha$.
\end{proof}

In the next two propositions we study
the topology of $(H_+)_{\{s\}}$, which is needed
because we construct the extension on leaves $(H_+)_{\{s\}}$ separately.

\begin{prop} \label{prop:formgeom}
Let $C$ be an $m \times m$ symmetric nondegenerate
real matrix, with $k$ positive and $\ell$ negative eigenvalues $(m =
k+\ell)$.
Given $s \in \R$,
consider the manifold with boundary
\begin{equation}
X_s = \{ x\in \R^m \mid s \geq x^t C x \} .
\end{equation}
Suppose that $k \geq 2$.
\begin{enumerate}[(i)]
\item
If either $s > 0$, or $s < 0$ but $\ell \geq 3$, then $X_s$ is
connected and also simply connected, and $\partial X_s$ is connected.
\item
If $s < 0$ but $\ell = 2$, then $X_s$ and $\partial X_s$ are
connected, $\pi_1(X_s) = \mathbb{Z}$, and the generator
$\gamma \colon S^1 \to X_s$ lies in $\partial X_s$.
In fact, after the change of coordinates \eqref{eq:Cnewcoords}
it is the set given by $x_1 = \dots = x_k = 0$, $x_{m-1}^2+x_{m}^2 = -s$.
\item
If $s < 0$ and $\ell = 1$, then $X_s$ has exactly two simply connected
components, each of which has connected boundary.  
\end{enumerate}
\end{prop}

\begin{proof}
Sylvester's law of inertia shows that
without loss of generality we suppose
\begin{equation} \label{eq:Cnewcoords}
x^tCx = x_1^2 + \dots + x_k^2 - x_{k+1}^2 - \dots - x_{m}^2
\end{equation}
Therefore,
write
\begin{align}
& X_s = \{ x \in \R^m \mid x_1^2 + \dots + x_k^2 - x_{k+1}^2 - \dots -
x_{m}^2 \leq s \} , \\
& Y_s = \partial X_s = \{ x \in \R^m \mid x_1^2 + \dots + x_k^2 - x_{k+1}^2 - \dots -
x_{m}^2 = s \} .
\end{align}
If $\ell = 0$, then the set $X_s$ is a ball, in which case the proposition
is trivial, so for the rest of the proof suppose $\ell \geq 1$.

Via a rotation
we continuously deform $x$ so that $x_1 > 0$,
$x_2 = \dots = x_k = 0$, and $x_{k+1} \not= 0$, $x_{k+2} = \dots = x_m = 0$,
while still staying in $Y_s$.  If $\ell \geq 2$, then we further
continuously rotate until we also get $x_{k+1} > 0$.

If $s > 0$, we let $x_{k+1}$ continuously go to
0, while $x_1$ goes to $\sqrt{s}$.  That is, the set $Y_s$ is connected.

While if $s < 0$, we let $x_1$ go to
zero and $x_{k+1}$ go to either $\sqrt{-s}$ or $-\sqrt{-s}$.  The second is
a possibility only if $\ell = 1$, and it is clear that in this case $Y_s$
has 
exactly two components, as there is no way to pass through a point
where $x_{k+1} = 0$ if $s < 0$.

In either case, if $Y_s$ is connected then $X_s$ must be connected as it is
a manifold with boundary and $Y_s = \partial X_s$.  Similarly if $\ell=1$
and $s < 0$, we
find that there must be exactly two components of $X_s$ as there is no way
to pass through the point $x_{k+1}$ and still stay in the set $X_s$.

Let us consider the fundamental group of $X_s$.
Suppose $\gamma \colon S^1 \to X_s$ is a loop in $X_s$.
The set $X_s$ is given by
$x_1^2 + \dots + x_k^2 - x_{k+1}^2 - \dots - x_{m}^2 \leq s$.
We continuously make the first $k$ components of the loop go 0, while
staying in $X_s$.  Now the loop is in the set $Z$ defined by $x_1 = \dots = x_k = 0$
and $- x_{k+1}^2 - \dots - x_{m}^2 \leq s$.  The set $Z$ is $\R^\ell$
if $s > 0$, meaning the loop deforms to a point and $X_s$ is simply
connected.

If $s < 0$ then the set
given by
$- x_{k+1}^2 - \dots - x_{m}^2 \leq s$ and
$x_1 = \dots = x_k = 0$, is $\R^\ell \setminus B(0,\sqrt{-s})$,
which is simply connected if $\ell \neq 2$ (it has 2 components if $\ell=1$).
If $\ell = 2$, clearly then $\pi_1(X_s) = \mathbb{Z}$
as any loop in $\R^2 \setminus B(0,\sqrt{-s})$ is homotopic to a multiple
of the circle.
Consequently, the generator is given by
$- x_{m-1}^2 - x_{m}^2 = s$ as claimed.
\end{proof}

\begin{prop} \label{prop:geometryofleaves}
Suppose $M$ and $H_+$ are defined near the origin
by \eqref{eq:basicnormdiagA} and \eqref{eq:basicnormH}, $a \geq 2$,
and $Q$ is nondegenerate.

Then there exists a neighborhood $U \subset \C^n \times \R$
of the origin such that for all $s$ the manifold with boundary
${(H_+ \cap U)}_{\{s\}}$
either
\begin{enumerate}[(i)]
\item is empty,
\item 
is connected with connected boundary, or
\item
has two components, each one with connected boundary, in which case $s < 0$.
\end{enumerate}
Each component is simply connected, or it has a single generator
of the first fundamental group, which is a loop in $M \cap U$.
\end{prop}

\begin{proof}
Write $M$ as
\begin{equation} \label{eq:MBmatrix}
s = \sum_{j=1}^a \sabs{z_j}^2 - \sum_{j=a+1}^{a+b} \sabs{z_j}^2 + z^t B z +
\overline{ z^t B z} + E(z,\bar{z}) 
\end{equation}
for $E$ in $O(3)$, where $B$ is a complex symmetric matrix and $a \geq 2$.
Via an
$n$-by-$n$ unitary matrix of the form $T \oplus I_{n-2}$,
where $T$ is a 2-by-2 unitary matrix and $I_{n-2}$ is the identity matrix,
which keeps $A$ fixed, we put the symmetric matrix $B$
into the form
\begin{equation}
\begin{bmatrix}
\lambda_1 & 0         & * \\
0         & \lambda_2 & * \\
*       & *       & *
\end{bmatrix} ,
\end{equation}
where $\lambda_j \geq 0$.
Setting $z_3 = \dots = z_n = 0$, we find that $Q$ becomes a real quadratic form
\begin{equation}
\sabs{z_1}^2 + \sabs{z_2}^2 + \lambda_1 z_1^2 + \lambda_2 z_2^2 .
\end{equation}
By direct computation, the symmetric real $4 \times 4$ matrix representing
this quadratic form has at least 2 positive eigenvalues.  Therefore the real
$2n \times 2n$ matrix representing $Q$ has at least 2 positive eigenvalues.
Hence the set $s \geq Q(z,\bar{z})$ satisfies the 
hypotheses of Proposition~\ref{prop:formgeom}, and so it satisfies
the conclusion of the proposition.

If $E$ is not zero, we note that since $Q$ is nondegenerate, the Morse lemma
implies that there exists a
real smooth change of coordinates at the origin in $\C^n$ such that
$Q(z,\bar{z}) + E(z,\bar{z})$ becomes a quadratic form in the new
coordinates.  The result again follows from
Proposition~\ref{prop:formgeom}.
\end{proof}

We now show the existence of affine analytic discs through any point
in a small neighborhood of zero in $H_+$ that are attached to the CR
points of $M$.  Such discs either continuously shrink towards a CR point
of $M$, or continuously deform to discs through $(0,s)$ and living only
in the first two coordinates of $z$.

\begin{lemma} \label{lemma:affinedisks}
Suppose $M$ and $H_+$ are defined near the origin
by \eqref{eq:basicnormdiagA} and \eqref{eq:basicnormH}, $a \geq 2$,
and $Q$ is nondegenerate.

Given any neighborhood $V$ of the origin, there exist neighborhoods
$U_1 \subset U_2 \subset V$ of the origin
such that, for every point
$(z,s) \in (H_+ \setminus M) \cap U_1$,
there exists a complex affine function
$L \colon \C \to \C^n \times \{ s \}$,
such that
$L(0) = (z,s)$,
$L^{-1}(M \cap U_2)$ is compact and connected, and such that $L(\C)$ intersects
$M \cap U_2$ transversally and intersects only CR points of $M$.

Furthermore, if $s \not= 0$, and $W$ is a neighborhood of $M \setminus \{ 0 \}$, then there exists a continuous family
of such affine discs, that is a continuous family of affine
functions $L_t \colon \C \to \C^n \times \{ s \}$, $t \in [0,1]$,
such that $L_t^{-1}(M \cap U_2)$ is compact for all $t$,
$L_1 = L$, and $L_0(\C) \cap U_2 \cap H_+ \subset W$,
or otherwise $L_0(0) = (0,s)$ with $z_3=\dots=z_n=0$ on $L_0(\C)$
(and also $s > 0$).
\end{lemma}

The fact that $L(\C) \cap H_+ \cap U_2$ are discs
because $L(\C) \cap M \cap U_2$ is connected is not strictly necessary for
our purposes, but it simplifies somewhat the terminology and proofs.

\begin{proof}
To simplify notation, we assume below that we can make $V$
smaller so that $U_2 = V$.  We also write $U$ instead of $U_1$.
Using coordinates $(\zeta,\tau) \in \C^n \times \R$,
suppose $M \subset \C^n \times \R$ is given by
\begin{equation} \label{eq:Mtauzeta}
\tau = \sum_{j=1}^a \sabs{\zeta_j}^2 - \sum_{j=a+1}^{a+b} \sabs{\zeta_j}^2 + B(\zeta,\zeta)
+ \overline{B(\zeta,\zeta)} + E(\zeta,\bar{\zeta}) ,
\end{equation}
for $E$ in $O(3)$, where $a \geq 2$.

As $B$ is given by a complex symmetric matrix,
then as above, we put the matrix representing $B$ into the form
\begin{equation}
\begin{bmatrix}
\lambda_1 & 0         & * \\
0         & \lambda_2 & * \\
*       & *       & *
\end{bmatrix} ,
\end{equation}
where $\lambda_j \geq 0$.

Then $H_+\setminus M$ is given by
\begin{equation} \label{affdiskeqn}
\tau > \sum_{j=1}^a \sabs{\zeta_j}^2 - \sum_{j=a+1}^{a+b} \sabs{\zeta_j}^2 + B(\zeta,\zeta)
+ \overline{B(\zeta,\zeta)} + E(\zeta,\bar{\zeta}) .
\end{equation}
Take some $(z,s) \in H_+ \setminus M$.
Let $\ell \colon \C \to \C^n$ be given by
\begin{equation}
\ell(\xi) = \left( z_1 + c_1 \xi , z_2 + c_2 \xi , z_3 , \dots, z_n \right)
\end{equation}
and $L \colon \C \to \C^n\times \{s\}$ be given by
$L(\xi) = \bigl( \ell(\xi), s \bigr)$.

Let us plug $L$ into \eqref{affdiskeqn} to get
\begin{multline} \label{inequalityLpluggedin}
s > \sabs{z_1 + c_1 \xi}^2 + \sabs{z_2 + c_2 \xi}^2 + \sum_{j=3}^a
\sabs{z_j}^2 - \sum_{j=a+1}^{a+b} \sabs{z_j}^2
\\
+
B\bigl(\ell(\xi),\ell(\xi)\bigr)
+
\overline{B\bigl(\ell(\xi),\ell(\xi)\bigr)}
+
E\bigl(\ell(\xi),\overline{\ell(\xi)}\bigr) .
\end{multline}
We expand the quadratic terms to get
\begin{equation}
s > Q(z,\bar{z}) + P(z,\bar{z}) \xi + \overline{P(z,\bar{z})} \bar{\xi}
+ \alpha \xi^2 + \bar{\alpha} \bar{\xi}^2
+  \bigl(\abs{c_1}^2+\abs{c_2}^2 \bigr) \xi \bar{\xi} +
E\bigl(\ell(\xi),\overline{\ell(\xi)}\bigr) ,
\end{equation}
where $Q$ is a real quadratic form, $P$ is real-linear, and
\begin{equation}
\alpha = \lambda_1 c_1^2 + \lambda_2 c_2^2 .
\end{equation}
We find one solution $c_1,c_2$ to $\alpha = 0$, $\abs{c_1}^2+\abs{c_2}^2=1$.
Because
the coefficient in front of $\xi \bar{\xi}$ is 1, we rewrite
the inequality as
\begin{equation}
s >
\abs{P(z,\bar{z}) + \xi}^2 - R(z,\bar{z})
+
E\bigl(\ell(\xi),\overline{\ell(\xi)}\bigr) ,
\end{equation}
where
$R$ is a real quadratic form.
As $E$ is of order 3, by making $V$ smaller if necessary, we can without
loss of generality restrict $z$ (and therefore also $\xi$) to a small
neighborhood such that,
\begin{equation}
\abs{
E\bigl(\ell(\xi),\overline{\ell(\xi)}\bigr)
}
\leq
\frac{1}{2}\snorm{z}^2 + \frac{1}{2} \abs{P(z,\bar{z}) + \xi}^2 .
\end{equation}
We obtain
\begin{equation} \label{estimateforPxi}
s + R(z,\bar{z})
+ \frac{1}{2} \snorm{z}^2 
>
\frac{1}{2} \abs{P(z,\bar{z}) + \xi}^2 .
\end{equation}
The inequality \eqref{estimateforPxi} is satisfied at $\xi=0$.
The left hand side of
\eqref{estimateforPxi} does not depend on $\xi$ and goes to zero 
as $(z,s) \to (0,0)$ in $H_+$.  Therefore, picking a small
enough neighborhood $U$, if $(z,s) \in U \cap H_+$, then
$L(\xi)$ is in $H_+$
only for points well inside $V$.
In other words, equality in \eqref{inequalityLpluggedin} is satisfied only
in $V$, and so $L(\C)\cap M \cap V$ is compact,
which is what we wanted.

If $s \not=0$, then the origin is not in $L(\C)$. As there are
infinitely many solutions $c_1,c_2$, and so infinitely many
possible lines $L$, we pick one that does not go
through the origin even if $s=0$.  Therefore, the line intersects only
at CR points, and we make the intersection transversal
by applying Proposition~\ref{prop:transversal}.

We still need to show that
$L(\C)\cap M \cap V$ is connected so that we obtain analytic discs rather
than more general one dimensional manifolds with boundary.
Write \eqref{eq:Mtauzeta} as $\tau = \rho(\zeta,\bar{\zeta})$ and define
\begin{equation}
r(\xi,\bar{\xi}) = \rho\bigl(\ell(\xi),\overline{\ell(\xi)}\bigr) 
\end{equation}
The right hand side of \eqref{inequalityLpluggedin}
is $r(\xi,\bar{\xi})$, so
if $V$ is small enough, $r$ is subharmonic.
We also assume $V$ is convex.

Suppose for a contradiction that 
$L(\C)\cap M \cap V$ is disconnected.  It is a one real dimensional smooth
curve as the intersection is transversal.  It is therefore composed of
several curves, one of which say $C_0$ is the ``outside'' curve, and some curves
inside, say one of these is $C_1$.
As $V$ is convex, the interior of $C_0$ is contained within $L(\C) \cap V$.
Write \eqref{eq:Mtauzeta} as $\tau = \rho(\zeta,\bar{\zeta})$ and define
\begin{equation}
r(\xi,\bar{\xi}) = \rho\bigl(\ell(\xi),\overline{\ell(\xi)}\bigr) .
\end{equation}
If $r < s$ on both sides of $C_1$, along which $r=s$,
we would violate the maximum principle.
Therefore $r > s$ on at least some points
inside $C_1$, and so $r$ achieves a maximum inside $C_1$,
which is again a contradiction as $r$ is subharmonic.
Consequently, $C_1$ did not
exist and $L(\C) \cap M \cap V$ is connected.

For the final result we replace $z$ with $tz$ in the above estimates
and write $L_t(\xi) = (tz_1 + c_1 \xi,
tz_2 + c_2 \xi,
tz_3,\dots,tz_n,s)$.
Then $L_1 = L$ as claimed, and $L_t$ is a continuous family.  Clearly
$(tz,s) \in U$ for all $t\in [0,1]$, so $L_t(\C) \cap M \cap V$
is compact.
If $s < 0$, then in \eqref{estimateforPxi} we see that the
left hand side must become negative before $t$ reaches zero.
So for some $t$, the set $L_t(\C) \cap H_+ \cap V$ is empty,
and hence for some slightly larger $t$,
the set $L_t(\C) \cap H_+ \cap V \subset W$.
If $s > 0$, and  $L_t(\C) \cap H_+ \cap V$ is never empty for
$t \in [0,1]$, then $L_0(0) = (0,s)$.
\end{proof}


\section{The extension near a CR singularity} \label{section:extnnearsing}

Using the affine discs we first show interior regularity of an extension,
if it exists.

\begin{lemma} \label{lemma:intreg}
Suppose $\Omega \subset \C^n\times \R$ is a domain with smooth boundary
and $f \colon \partial \Omega \to \C$
is a smooth function that is CR at CR points of $\partial \Omega$.  Suppose
for some $s_0 \in \R$,
$L \colon \C \to \C^n \times \{ s_0 \} \subset \C^n \times \R$
is a complex affine mapping, and
$D$ is a bounded component of $L(\C) \cap \Omega$.
Suppose $U \subset \C^n \times \R$ is a neighborhood of $\overline{D}$
and $F \colon U \cap \Omega \to \C$ is such that
for any fixed $s$ near $s_0$, $z \mapsto F(z,s)$
is holomorphic and
extends continuously up to the boundary $(\partial \Omega \cap U)_{\{s\}}$,
where it agrees with $f$.
Then $F$ is smooth in a neighborhood of $D$ in $\Omega$.
\end{lemma}

\begin{proof}
Making $\Omega$ smaller if necessary,
we assume $\Omega$ is bounded and $D = L(\C) \cap \Omega$.
Pick some point $(z_0,s_0) \in L(\C) \cap \Omega$, where
$L(\xi) = (c \xi + z_0,s_0)$ and $c \in \C^n$.
After a small perturbation of $c$, we assume $L(\C)$
intersects $\partial \Omega$ transversally via
Proposition~\ref{prop:transversal}.
Without loss of generality assume $c = (1,0,\dots,0)$ and $z_0 = 0$,
so $L(\xi) = (\xi,0,\dots,0,s_0)$.
Write $L_{z',s}(\xi) = (\xi,z',s)$,
where $z' \in \C^{n-1}$.
As the intersection $L(\C) \cap \partial \Omega$ is transversal, 
for all $s$ near $s_0$ and $z'$ near 0, 
$L_{z',s}(\C) \cap \Omega$ is bounded and connected, and
$L_{z',s}(\C) \cap \partial \Omega$ is a smooth path.
We write
$F$ using the Cauchy integral formula as an integral over this path.
That is, for $(z,s) = (\xi,z',s)$ near $(z_0,s_0)$,
\begin{equation}
F(z,s) = F\bigl(L_{z',s}(\xi)\bigr)
= \frac{1}{2\pi i} \int_{L_{z',s}^{-1}(\partial \Omega)}
\frac{f\bigl(L_{z',s}(\tau)\bigr)}{\tau-\xi} \, d\tau .
\end{equation}
The intersection $L_{z',s}(\C) \cap \partial \Omega$ is transversal
if we change $z'$ and $s$ slightly.
Therefore $L_{z',s}^{-1}(\partial \Omega)$ varies smoothly with $(z',s)$
and hence with $(z,s)$.  Similarly $\xi$ varies smoothly
as a function of $z$.  Therefore, the function $F$ is smooth in $(z,s)$
near $(z_0,s_0)$ and the lemma follows.
\end{proof}

\begin{lemma} \label{lemma:extW}
Suppose $M$ and $H_+$ are defined near the origin
by \eqref{eq:basicnormdiagA} and \eqref{eq:basicnormH}, $a \geq 2$,
and $Q$ is nondegenerate.
Then there exists a neighborhood $W$ of $M\setminus\{0\}$ such that smooth
CR functions on $M$ extend to CR functions on $W\cap H_+$.

That is, given a $C^\infty$ function $f \colon M \to \C$ which is CR
outside the origin, there exists a function $F \colon W \cap H_+ \to \C$,
such that the
restrictions to leaves, $z \mapsto F(z,s)$, are holomorphic
and continuous up to $(W \cap M)_{\{s\}}$ where it agrees with $f$.
\end{lemma}

\begin{proof}
Let $M$ be defined by $0 = - s + A(z,\bar{z}) + B(z,z) +
\overline{B(z,z)}  + E(z,\bar{z})$.
We use the same function to define $(M)_{\{s\}}$.
For points $(z,s)$ near the origin,
the Levi-form of $(M)_{\{s\}}$ is a small
perturbation of the form $A$ restricted to the $T^{(1,0)}_{(z,s)} (M)_{\{s\}}$.
As $A$ has two positive eigenvalues, then the Levi-form of
$(M)_{\{s\}}$ must have at least one positive eigenvalue.  This eigenvalue
corresponds to the side $(H_+)_{\{s\}}$.  Therefore
for some neighborhood of the origin, we can apply
Theorem~\ref{thm:extCR} near all CR points of $M$.
The extension is unique and so near CR points of $M$ it can be patched together.
In other words, shrinking $M$ and $H_+$ to a smaller neighborhood of the
origin if needed, we have an extension to some neighborhood of $M \setminus
\{ 0 \}$ in $H_+$.
That is,
we know that there are a neighborhood $W$ of $M \setminus \{ 0 \}$,
and an $F$
defined on $W \cap H_+$, holomorphic along leaves, and continuous along
leaves up to $M$.
\end{proof}

Extension for $s < 0$ follows using the affine
discs; however, for certain points where $s > 0$,
it is necessary use different
complex manifolds with boundary attached to $M$.

\begin{lemma} \label{lemma:extanuli}
Suppose $M$ and $H_+$ are defined near the origin
by \eqref{eq:basicnormdiagA} and \eqref{eq:basicnormH}, $a \geq 2$,
and $Q$ is nondegenerate. Let $W$ be as in Lemma~\ref{lemma:extW}.

Then there exists a neighborhood $U$ of the origin
with the following property.
Each point $(z,s) \in U$, where $s > 0$ and $z_3 = \dots = z_n = 0$,
is connected via a path in $(U \cap H_+ \setminus M)_{\{s\}}$ to a point
$(z',s) \in W$.

Additionally, if $f \colon M \to \C$ is a $C^\infty$ function
that is CR outside the origin, then
$F$ extends in $\C^n \times \{ s \}$
by analytic continuation along the above paths.
\end{lemma}

\begin{proof}
Without loss of generality it is enough to consider $n=2$. Let 
$M \subset \C^2 \times \R$ be given by
\begin{equation}
s = \abs{z_1}^2 + \abs{z_2}^2 + 2\Re\bigl(\lambda_1z_1^2 +
\lambda_2z_2^2\bigr) + E(z,\bar{z}) .
\end{equation}
We suppose that $s > 0$.

If $\lambda_1 = \lambda_2 = 0$, then the set $(H_+)_{\{s\}}$ is a ball of
radius $\sqrt{s}$,
and the extension follows by the standard Hartogs extension theorem.
So suppose that at least one $\lambda_j$ is not zero.

Let $(z,s) \in H_+ \setminus M$ be some fixed point, where
$\lambda_1z_1^2 + \lambda_2z_2^2 \not= 0$.
We work in $\C^2$, that is on one fixed leaf.
Let
\begin{equation}
g(\zeta_1,\zeta_2) = \lambda_1\zeta_1^2 + \lambda_2\zeta_2^2.
\end{equation}
We consider the one dimensional submanifolds
$X_t \subset \C^2$ given by $g(\zeta) = t$ for nonzero $t$.
Let $t_0 = g(z)$ (note that $t_0 \not= 0$).

Let us consider the set of points on $X_t$ that correspond
to $H_+ \setminus M$, that is let us
consider the set of points where $\zeta \in X_t$ and
\begin{equation}
s > \snorm{\zeta}^2 + 2\Re g(\zeta) + E(\zeta,\bar{\zeta}) .
\end{equation}
Let us call this set of points $Y_t \subset X_t$.
In other words
\begin{equation} \label{eq:ineqanuli}
Y_t = \{\zeta\in X_t\,:\,s-2\Re t >
\snorm{\zeta}^2 + E(\zeta,\bar{\zeta}) \}.
\end{equation}

As $E(\zeta,\bar{\zeta})$ is $o(\snorm{\zeta}^2)$, we pick
a neighborhood $U$ in which 
\begin{equation}
\abs{E(\zeta,\bar{\zeta})} < \frac{1}{2}\snorm{\zeta}^2 .
\end{equation}
Then \eqref{eq:ineqanuli} implies
\begin{equation}
2(s-2\Re t) > \snorm{\zeta}^2 .
\end{equation}
We obtain that $Y_t$ is a relatively compact subset of $X_t$.
Denote by $\partial Y_t$ the relative boundary of $Y_t$ in $X_t$.
If the neighborhood
$U$ is picked small enough, that is if $(z,s)$ is picked to be
close enough to the origin
then the boundary $\partial Y_{t_0}$ lies on $M_{\{s\}}$,
and it also lies on
$M_{\{s\}}$ (or is empty)
for all $t$ such that $\Re t \geq \Re t_0$. 

Therefore we find a continuous family of $Y_t$ such that
$\partial Y_t \subset M_{\{s\}}$,
and as we
move $t$ to make $\Re t$ larger, $Y_t$ must be empty when $2(s-2\Re t)=0$.
Hence for some $t$ the set $Y_t$ is nonempty and completely within $W_{\{s\}}$.
Now move $t$
towards $t_0$ along some path. Let $t_1$ be the first such $t$ where
$Y_{t_1}$ contains at least one point
not in $W_{\{s\}}$.  Clearly such points must be in the interior and since
$W_{\{s\}}$ is a
neighborhood of $M_{\{s\}}$, $Y_{t_1} \setminus W_{\{s\}}$
is a compact set in $Y_{t_1}$.
Let $\varphi_t\colon \C\setminus \{0\} \to X_t$ be the natural rational
parametrization
of $X_t$.  For example, if $\lambda_j \not= 0$, it is
\begin{equation*}
\varphi_t(\xi) = 
\left(
\frac{\xi+\frac{t}{\xi}}{2 \sqrt{\lambda_1}},
\frac{\xi-\frac{t}{\xi}}{2i \sqrt{\lambda_2}}
\right) .
\end{equation*}
In particular, $\varphi_t$ varies analytically with $t$.
There exists a smooth path
$\Gamma \subset \varphi_{t_1}^{-1}(Y_{t_1}) \subset\C \setminus \{0 \}$
that goes exactly once around
$\varphi_{t_1}^{-1}(Y_{t_1} \setminus W_{\{s\}})$.
We apply the Cauchy formula on $X_t$ for $t$ slightly before
getting to $t_1$ (in the sense of moving towards $t_0$),
so that $\Gamma \subset \varphi_t^{-1}(Y_t)$, and $z \in Y_t$:
\begin{equation}
F(z,s) = \frac{1}{2\pi i} \int_{\Gamma}
\frac{F\bigl(\varphi_{t}(\tau)\bigr)}{\tau-\varphi_{t}^{-1}(z)} \, d\tau.
\end{equation}

The formula therefore holds for $t$ in a neighborhood of $t_1$,
and $z \in Y_t \cap W_{\{s\}}$, and hence we define an extension of $F$
in all of $Y_t$.

Hence we have the required extension $F$ to all points $(z,s)$
in some neighborhood $U$ except
perhaps points where either $g(z) = 0$.
We find an extension $F$
into an open set minus the subvariety given by $g(z) = 0$.  However, we
notice that this subvariety, which is a union of two complex lines through
the origin, must in fact intersect $H_+$ only in a bounded set, which is
clear from \eqref{eq:ineqanuli}.  Thus the set to which we did not yet
extend $F$ is a compact subset of the leaf and we use the standard Hartogs
extension phenomenon.
\end{proof}

\begin{lemma} \label{lemma:extallbutzero}
Suppose $M$ and $H_+$ are defined near the origin
by \eqref{eq:basicnormdiagA} and \eqref{eq:basicnormH}, $a \geq 2$,
and $Q$ is nondegenerate.

Then there exists a neighborhood $U$ of the origin with the following property.
If $f \colon M \to \C$ is $C^\infty$ and CR outside the origin,
then there is a function $F \in C^\infty(H_+ \cap U \setminus \{ 0 \})$
such that
$F$ is CR on $(H_+ \setminus M) \cap U$ and $F|_{M \cap U} = f|_{M \cap U}$.
\end{lemma}

\begin{proof}
We apply Lemma~\ref{lemma:extW} to obtain $W$ and define $F$ in $W$,
and we also apply Lemma~\ref{lemma:extanuli} to show that we
can analytically extend $F$ to all points $(z,s)$ where $s > 0$
and $z_3=\dots=z_n = 0$.

We must show that for some neighborhood $U$, for each $s$
and for every point $z$ in $(U \cap H_+ \setminus M)_{\{s\}}$,
we can extend via analytic continuation from some point
$w \in (U \cap W)_{\{s\}}$.
We then show that this extension is unique, as long as $U$ is small
enough.

We find a small enough neighborhood $V$ where
Proposition~\ref{prop:geometryofleaves} applies, and we assume that
$M$ and $H_+$ are closed submanifolds of $V$.

Let us suppose
that $s \not=0$.  Via Lemma~\ref{lemma:affinedisks} there are neighborhoods
$U_1$ and $U_2 \subset V$
such that for any given $(z,s) \in U_1$,
we have a family of
affine maps $L_t$, with $L_1(0) = (z,s)$ and such that the image
$L_t(\C) \cap U_2 \cap H_+$
either ends up in $W$ for $t=0$, or for $t=0$ it ends up in
the set $z_3 = \dots = z_n = 0$.
If the images of $L_t$ end up in $W$, then we may apply the
Kontinuit\"{a}tssatz (see, e.g., \cite{Shabat:book}*{page 189}).
As there exists a holomorphic function in a
neighborhood of 
$(L_0(\C) \cap U_2 \cap H_+ \setminus M)_{\{s\}}$ in $\C^n$,
there exists a holomorphic function in a neighborhood of 
$(L_1(\C) \cap U_2 \cap H_+ \setminus M)_{\{s\}}$.
Hence the function extends via analytic continuation to $(z,s)$.
We let $U = U_1$.

We need to now show that the extension is single valued.
We pick $U$ such that Proposition~\ref{prop:geometryofleaves}
applies there.
Each component (possibly
two) of $(U \cap H_+)_{\{s\}}$ is a submanifold with boundary whose boundary
(the submanifold $M_{\{s\}}$) is connected.
The proposition also says that $(U \cap H_+)_{\{s\}}$ and therefore
$(U \cap H_+ \setminus M)_{\{s\}}$ is either simply
connected, or the fundamental group is $\mathbb{Z}$ whose generator is a
loop in the boundary $U \cap M_{\{s\}}$.  If the neighborhood is simply
connected the extension $F$ is single valued.  If it is not simply
connected, we push the generator of the fundamental group
from the boundary into
$(U \cap H_+ \setminus M)_{\{s\}}$, but such that it still stays inside $W$.
The function is therefore already single valued on the generator of the
fundamental group.  Therefore, the extension is single valued in
$(U \cap H_+)_{\{s\}}$.

Finally suppose that $s = 0$.
The first part of Lemma~\ref{lemma:affinedisks} applies
in $U$ even for $s=0$.  So if
$(z,0) \in (H_+ \setminus M) \cap U$, then there is an affine disc
through $(z,0)$ attached to $M_{\{0\}}$ and not
attached to the CR singularity.
The boundary of this disc falls into $W$, that is we have an
extension near the
boundary of this disc.

As an extension exists in a neighborhood of the CR points,
even on the $s=0$ leaf,
we take a slightly smaller disc $D_0$ such that
$\partial D_0 \subset (W \cap H_+ \setminus M)_{\{0\}}$,
that is $F$ is holomorphic in a neighborhood of
$\partial D_0$.  Let $D_s$ be the identical disc, but 
on the $s$-leaf rather than the $0$ leaf.
For small $s \not = 0$,
$\partial D_s \subset (U \cap H_+ \setminus M)_{\{s\}}$,
and we know that $F$ was extended above to the entire disc
$\overline{D_s}$.
Denote by $L_0$ and $L_s$ the respective
affine functions.

We define an extension $\tilde{F}$ in $D_0$ via
the Cauchy integral formula:
\begin{equation}
\tilde{F}(z,0) = \frac{1}{2\pi i} \int_{L_0^{-1}(\partial D_0)}
\frac{F\bigl(L_0(\tau)\bigr)}{\tau-L_0^{-1}(z)} d\tau
\end{equation}
As $F$ on $D_s$ is given by
\begin{equation}
F(z,s) =  \frac{1}{2\pi i} \int_{L_s^{-1}(\partial D_s)}
\frac{F\bigl(L_s(\tau)\bigr)}{\tau-L_s^{-1}(z)} d\tau
\end{equation}
by continuity, $\tilde{F}=F$ where both are defined.  Hence $F$
extends through the $s=0$ leaf in some neighborhood.

We have an extension $F$ defined in
$H_+ \cap U$ for some small enough neighborhood $U$
of the origin.
Regularity of $F$ the function in $(H_+ \setminus M) \cap U$
follows via Lemma~\ref{lemma:intreg}.
Regularity near the CR points of
$M$, that is regularity on $(H_+ \setminus \{0\}) \cap U$
follows from Theorem~\ref{thm:extCR}.
\end{proof}


\section{Formal extension at a CR singularity} \label{section:formal}

The formal extension in \cite{crext2} is stated for a nondegenerate $A$;
however, it also works as long as $A$ has at least two positive eigenvalues
and $Q$ is nondegenerate.  For completeness we give the statement
in the notation that we need and a sketch of the proof.

\begin{lemma}\label{lem:PolyExtn}
Suppose $M \subset \C^{n} \times \R$, $n \geq 2$, given by
\begin{equation} \label{eq:basicmodel}
M: \quad s = \sum_{j=1}^a \sabs{z_j}^2 - \sum_{j=a+1}^{a+b} \sabs{z_j}^2 + B(z,z)
+ \overline{B(z,z)} ,
\end{equation}
$a \geq 2$, and $Q$ is nondegenerate.
Suppose $f(z,\bar{z})$ is a polynomial such that when considered as a
function on $M$ (parametrized by $z$), $f$ is a CR function on
$M_{CR}$.

Then there exists a polynomial $F(z,s)$ such that $f$ and $F$
agree on $M$, that~is,
\begin{equation}
f(z,\bar{z}) = F\bigl(z, A(z,\bar{z}) + B(z,z) + \overline{B(z,z)} \bigr) .
\end{equation}
Furthermore, if $f$ is homogeneous of degree $d$,
then $F$ is weighted homogeneous of degree $d$, that is,
\begin{equation}
F(z,s)=\sum\limits_{j+2k=d}\, P_j(z)s^k
\end{equation}
where $P_j$ is a homogeneous polynomial of degree $j$.
\end{lemma}

The proof is a combination of the proofs of 
Lemma~3.2 and Lemma~3.3 from \cite{crext2}.  The main difficulties
in \cite{crext2} are 1) handling the case where $A$ has one positive and
one negative eigenvalue, and 2) handling the case where the
CR singularity is large.
In the case of the above lemma, neither issue arises.

\begin{proof}[Sketch of Proof]
Restricting to $z_3 = \dots = z_n = 0$, we find a submanifold in
$\C^2 \times \R$ with positive definite $A$.
Let us therefore for the moment consider
$M' \subset \C^2 \times \R$ given by
\begin{equation}
M': \quad s = \sabs{z_1}^2 + \sabs{z_2}^2 + B'(z,z) + \overline{B'(z,z)} .
\end{equation}
In \cite{crext2} we noted
that there exists a so-called elliptic direction $c \in \C^2$,
in particular,
for $s > 0$ 
the image of the map
\begin{equation}
\xi \mapsto (c \xi,s)
\end{equation}
intersects $M'$ in an ellipse and therefore induces an analytic disc.
Let us call this disc $\Delta_{c,s}$.
We will again apply the
Kontinuit\"atssatz but in this case we are allowed to move between
leaves as all the data ($M$ and $f$) is analytic (in fact polynomial).

For small $v \in \C^2$
\begin{equation}
\xi \mapsto (c \xi+v,s)
\end{equation}
still meets $M'$ in an ellipse, and thus induces a disc
$\Delta_{c,s,v}$ attached to $M'$.  Fixing such a $v$ and
letting $s$ go to zero, we find that eventually
$\Delta_{c,s,v}$
must become empty.  Therefore we find a family of affine
analytic discs attached to $M'$ that shrink down to a point on $M'$
that is not the origin.

We return to our original $M \subset \C^n \times \R$.
We abuse notation somewhat by writing $c$ instead of $(c,0) \in \C^n$,
and we still write $\Delta_{c,s} \subset \C^n \times \R$ for the analytic
disc above.

We have a family of analytic discs attached to $M_{CR}$
where one end of the family is a disc, $\Delta_{c,s}$,
through a point $(0,s)$
and the other end of the family shrinks to point on $M_{CR}$.

We now complexify $s$ to consider $M$ as a subset of $\C^{n+1}$.
As $M$ and $f$ are analytic and $f$ is CR it extends to a holomorphic
function of a neighborhood of $M_{CR}$ in $\C^{n+1}$.
The Kontinuit\"atssatz therefore implies that $f$ extends to a holomorphic
function of a neighborhood of $\overline{\Delta_{c,s}}$.  As
the discs are always attached to $M_{CR}$ we find that the extension agrees
with $f$ on the intersection $\overline{\Delta_{c,s}} \cap M$.

Note that $f$ therefore extends to a holomorphic function for a
neighborhood of discs $\overline{\Delta_{c',s}}$ for all $c' \in C \subset \C^n$
in a neighborhood $C$ of $c$.

Let $M_{c'} \subset \C \times \C$ be the manifold
defined by the pullback
\begin{equation}
(\xi,s) \mapsto (\xi c', s)
\end{equation}
In \cite{crext}*{Lemma 5.1}, we proved that a polynomial $P(\xi,\bar{\xi})$,
when considered as
a function on the manifold $M_{c'}$ parametrized by $\xi$,
extends to a polynomial in $\xi$ and $s$.

Thus for each $c'$ in $C$ we find a polynomial $F_{c'}(\xi,s)$
which extends $f(c' \xi, \overline{c' \xi})$.
Consider for a moment $\Delta_{c,1}$.
There exists a holomorphic function $F(z,s)$
defined in a neighborhood $\Delta_{c,1}$ that extends $f$.
For all $c' \in C$,
$F(c'\xi, s) = F_{c'}(\xi,s)$ on an open set, and since $C$
is an open set, then  $F(z,s)$ agrees with a polynomial on an open set.
See \cite{crext}*{Proposition 5.2}.
\end{proof}

The lemma implies the existence of a formal power series for an extension.
The following proposition and its proof is essentially the same as
Proposition 5.1 from \cite{crext2} with the necessary modifications
made for smooth functions rather than real-analytic functions.

\begin{prop}\label{prop:FormalPowerSeries}
Let $M \subset \C^{n} \times \R$, $n\ge 2$, be a smooth submanifold given by
\eqref{eq:basicnormdiagA}, $Q$ nondegenerate, $a \geq 2$,
and write \eqref{eq:basicnormdiagA}
as $s = \rho(z,\bar{z})$.

Suppose $f \in C^{\infty}(M)$
such that $f|_{M_{CR}}$ is a CR function.
There exists a formal power series $F(z,s)$ for $f$ at the origin,
that is, $F\bigl(z,\rho(z,\bar{z})\bigr) = f(z,\bar{z})$ formally at the
origin (parametrizing $M$ by $z$).
\end{prop}

\begin{proof}
Write $M$ as
\begin{equation}
s = Q(z,\bar{z}) + E(z,\bar{z}) ,
\end{equation}
where $E$ is $O(3)$ and $Q(z,\bar{z}) = A(z,\bar{z}) + B(z,z) +
\overline{B(z,z)}$.
Parametrizing $M$ by $z$, decompose $f$ as
\begin{equation}
f(z,\bar{z})
= f_k(z,\bar{z}) + \widetilde{f}(z,\bar{z}) ,
\end{equation}
where $f_k$ is a real-homogeneous polynomial of degree $k$
and $\widetilde{f}$ is $O(k+1)$.

The rest of the proof is essentially identical to the proof of
Proposition~5.1 from \cite{crext2}.  Let us go through it quickly.
A basis of CR vector fields for CR points of $M$ near the origin is
given by vector fields of the form $X =
\left(
Q_{\bar{z}_j}
+
E_{\bar{z}_j}
\right)
\frac{\partial}{\partial \bar{z}_k}
-
\left(
Q_{\bar{z}_k}
+
E_{\bar{z}_k}
\right)
\frac{\partial}{\partial \bar{z}_j}$,
and similarly $X^{quad} =
Q_{\bar{z}_j} \frac{\partial}{\partial \bar{z}_k}
-
Q_{\bar{z}_k} \frac{\partial}{\partial \bar{z}_j}$ for $M^{quad}$.
Then
\begin{equation}
0 = Xf =
X (f_k + \widetilde{f}) =
\left(
Q_{\bar{z}_j}
\right)
\frac{\partial f_k}{\partial \bar{z}_k}
-
\left(
Q_{\bar{z}_k}
\right)
\frac{\partial f_k}{\partial \bar{z}_j}
+
O(k+1)
=
X^{quad} f_k + O(k+1).
\end{equation}
Therefore $X^{quad} f_k = 0$ and
$f_k(z,\bar{z})$ is a CR function on the model $M^{quad}$.
By Lemma~\ref{lem:PolyExtn}, we write
$f_k(z,\bar{z}) = F_k\bigl(z,Q(z,\bar{z})\bigr)$ for some
weighted homogeneous $F_k(z,s)$.
The function $F_k(z,s)$ is CR on $M$, and if we
parametrize by $z$,
$F_k\bigl(z,\rho(z,\bar{z})\bigr)$ has
the same $k$th order terms as $f$, and
\begin{equation}
f(z,\bar{z}) - F_k\bigl(z,\rho(z,\bar{z})\bigr)
\end{equation}
is a CR function on $M$ vanishing to one higher order. 
We obtain a formal power series.
\end{proof}


\section{Regularity of the extension at a CR singularity} \label{section:regularityCRsing}

In this section we prove Theorem~\ref{thm:extCRsing}.
Let $M$, $H_+$, $f$ be as in the theorem.
Lemma~\ref{lemma:extallbutzero} gives us a neighborhood
$U$ and the extension $F$.
From now on, we assume $H_+ = U \cap H_+$ and
$M = U \cap M$.
In the following, we parametrize $M$ by $z$ as usual when writing
$f(z,\bar{z})$, and we compute the partial derivatives of $f$
with respect to this $z$.

\begin{claim}\label{claim:ExtnConts}
$F \in C(H_+)$.
\end{claim}

\begin{proof}
The extension $F \colon H_+ \to \C$ from Lemma~\ref{lemma:extallbutzero}
is smooth (and hence continuous) on $H_+ \setminus \{ 0 \}$.
By Lemma~\ref{lemma:affinedisks},
there exists a 
small neighborhood $U' \subset U$ of the origin
such that
for each point $(z,s) \in H_+ \cap U'$
there is an analytic disk $\Delta$ through $(z,s)$ with
$\partial \Delta \subset M$.
Via the maximum principle
\begin{equation}
\begin{split}
\abs{F(z,s)-f(0)}
& \leq
\sup\Big\{\abs{f(\zeta,\bar{\zeta})-f(0)}\,:\,(\zeta,s)\in \partial \Delta \Big\}
\\
&
\leq
\sup\Big\{\abs{f(\zeta,\bar{\zeta})-f(0)}\,:\,(\zeta,s)\in M\Big\}.
\end{split}
\end{equation}
As $s\to 0$, $M\ni(\zeta,s) \to 0$.
\end{proof}

The derivatives $F_{z_j}$ and $F_s$ extend smoothly to $H_+ \setminus \{ 0
\}$.
We need them to be smooth through the origin.
First, let us show that their
restrictions to $M$ extend smoothly through the origin.

\begin{claim}\label{claim:DerFSm}
$F_{z_j}|_M,F_s|_M \in C^\infty(M)$, for $1\leq j\leq n$.
\end{claim}

\begin{proof}
Let $s = \rho(z,\bar{z})$ define $M$, then write
\begin{equation}
\xi_j = \rho_{\bar{z}_j} =
\epsilon_j z_j + v_j \cdot \bar{z} + h.o.t.,
\end{equation}
for $\epsilon_j = -1,0,1$ and
some constant vector $v_j$, where $\cdot$ denotes dot product.
Similarly,
$\bar{\xi}_j = \rho_{z_j} =
\epsilon_j \bar{z}_j + \bar{v}_j \cdot z + h.o.t.$
Because $Q$ is nondegenerate we find that the $\xi_j$ and $\bar{\xi}_j$
are linearly independent and therefore $(\xi,\bar{\xi})$ gives
a smooth change of variables at
the origin.

We take derivatives outside the origin:
\begin{equation} \label{eq:fzbarjdivxij}
f_{\bar{z}_j} = F_s|_M \xi_j .
\end{equation}
The function $f_{\bar{z}_j}$ is smooth through the origin, so
the right hand side is smooth as well.  We need to show that
$f_{\bar{z}_j}$ is divisible by $\xi_j$ to show that $F_s|_M$ is
smooth.

A formal solution to the extension problem exists, so
the division is also true formally.  Thus for any order $m$, we write $f$ as
\begin{equation}
f(z,\bar{z}) = P\bigl(z,\rho(z,\bar{z})\bigr) + R(z,\bar{z}) ,
\end{equation}
where $P(z,s)$ is a polynomial of degree $m$
and $R$ is $O(m+1)$. So,
\begin{equation}
f_{\bar{z}_j} = P_s(z,\rho) \xi_j + R_{\bar{z}_j}(z,\bar{z}) .
\end{equation}
Hence, the lower order terms are not an obstruction to the division
of $f_{\bar{z}_j}$ by
$\xi_j$.

Because the
variables $\xi,\bar{\xi}$ are a smooth change of variables, we now
think of everything in terms of $\xi$ and $\bar{\xi}$.

Consider the real part of $\xi_j$ as a variable,
and apply the Malgrange-Mather
division theorem~\cite{Malgrange}*{Chapter V}.
One obtains smooth functions $q$ and
$r$, where
\begin{equation} \label{eq:fjdiv}
f_{\bar{z}_j} = q \xi_j + r ,
\end{equation}
and $r$ does not depend on the real part of $\xi_j$.
If $r$ had any finite order terms, the lowest order part of $r$ would be
divisible by $\xi_j$ as we saw above and would thus depend on the real part
of $\xi_j$.  Hence $r$ does not have any finite order terms.

Consider the ideal $I$ generated by $\xi_j$.
Outside the origin,
$f_{\bar{z}_j}$ is divisible by $\xi_j$ using \eqref{eq:fzbarjdivxij}.
Thus, $f_{\bar{z}_j}$ is
locally in the ideal at every point outside the origin.
At the origin, $r$ vanishes to infinite order in \eqref{eq:fjdiv},
and so $f_{\bar{z}_j}$ is in $I$ formally, as its Taylor series is a
Taylor series of a function that is in the ideal, namely $q\xi_j$.
As the Taylor series of $f_{\bar{z}_j}$ at each point belongs formally to
$I$, then
a theorem of Malgrange \cite{Malgrange}*{Theorem 1.1' in Chapter VI},
implies $f_{\bar{z}_j} \in I$.
So $f_{\bar{z}_j}$ is divisible by $\xi_j$, and hence $F_s|_M$
extends smoothly through the origin.

Write (on $M$)
\begin{equation}
f_{z_j} = F_{z_j}|_M + F_s|_M \bar{\xi}_j .
\end{equation}
As $F_s|_M$ is smooth and $f_{z_j}$ is smooth, then $F_{z_j}|_M$ is smooth
through the origin.
\end{proof}

The proof of Theorem~\ref{thm:extCRsing} now follows from the next claim.

\begin{claim} \label{claim:localsmext}
$F \in C^\infty(H_+)$.
\end{claim}

\begin{proof}
By Theorem~\ref{thm:extCR},
 $F\in C^\infty(H_+ \setminus \{ 0 \})$.
Let us first show $F \in C^{1}(H_+)$.
Above we saw that $F \in C(H_+)$ and the derivatives of
$F_{z_j}$ and $F_s$ restricted to $M$ are smooth functions.
They are also CR functions.
Then $F_{z_j}$ and $F_s$ satisfy the hypothesis of the theorem, and
so we apply what we know so far, and we find that $F_{z_j}$ and $F_s$
are continuous at the origin.  Hence $F \in C^1(H_+)$.
By iterating this procedure we find that $F \in C^\infty(H_+)$.
\end{proof}

\section{Further Examples} \label{section:furtherexamples}

\begin{example}
\emph{Extension can be one sided at every point.}

Let $M \subset \C^2 \times \R$ be given by
\begin{equation}
s = \abs{z_1}^2 + \abs{z_2}^2 = \snorm{z}^2 .
\end{equation}
Let $g \colon S^3 \subset \C^2 \to \C$ be a smooth CR function defined on the
sphere that does not extend holomorphically to the outside of the unit ball
through any
point of $S^3$; see \cite{Catlin} or \cite{Hakim-Sibony}.  Define
\begin{equation}
f(z,s) =
\begin{cases}
e^{-1/s^2} g\bigl(\frac{z}{\sqrt{s}}\bigr) & \text{if $s < 0$,}\\
0 & \text{if $s = 0$.}
\end{cases}
\end{equation}
Again it is easy to see that $f$ is CR.  Furthermore, on $M$ outside the
origin we find
\begin{equation}
f(z,\snorm{z}^2) =
e^{-1/\snorm{z}^4} g\left(\frac{z}{\snorm{z}}\right).
\end{equation}
Clearly $f$ is smooth outside the origin.
Since $S^3$ is compact all derivatives of $g$ are bounded.  Taking
derivatives of $f$, we see that all derivatives of $f$ are bounded near
the origin, and so $f$ is smooth at the origin as well.
The function extends to be smooth and CR on $H_+$ given by $s \geq \snorm{z}^2$.
However, for every $p \in M$, there is no neighborhood $U$ of $p$
such that $f$ extends to be CR on $U$.
\end{example}

\begin{example}
\emph{Extension fails in $n=1$.}

We already saw that the extension fails in $n=1$.  However, let us give
a further example in the smooth case.
Suppose $M \subset \C \times \R$ is a nonparabolic
Bishop surface given by
\begin{equation}
s = \abs{z}^2 + \lambda z^2 + \lambda \bar{z}^2 ,
\qquad
\text{(where $0 \leq \lambda < \infty$ and $\lambda \not= \tfrac{1}{2}$)}.
\end{equation}
This $M$ has a nondegenerate isolated CR singularity.  Define a smooth $f \colon \C \to \R$ that
is zero on the first quadrant of $\C$ and positive elsewhere.
Parametrizing $M$ by $z$, we have that $f(z,\bar{z})$ is a smooth function on $M$.
As the CR condition is vacuous, it can be considered a CR function.
For every $s\not= 0$, the leaf
\begin{equation}
(H_+)_{\{s\}}  = \{ z \in \C \mid s \geq \abs{z}^2 + \lambda z^2 + \lambda
\bar{z}^2 \}
\end{equation}
is either empty, or
has part of its boundary in the first quadrant.  Clearly 
the function $f$ cannot extend to this leaf holomorphically as it is not
identically zero, but it is zero on a nontrivial arc of the boundary.
\end{example}

The next two examples show that
the existence of the extension depends on the topology of the leaves $(H_+)_{\{s\}}$.

\begin{example}
\emph{Without two positive eigenvalues, leaves may have disconnected boundary.}

The submanifold $M$ given by
\begin{equation}
s=\abs{z_1}^2-\abs{z_2}^2+\lambda\left(z_1^2+\bar{z}_1^2\right)
\end{equation}
for $\lambda>1/2$,
has an isolated CR singularity.
Both $Q$ and $A$ are nondegenerate,
and $A$ has only one positive and one negative eigenvalue.
The leaves $(H_+)_{\{s\}}$ for $s > 0$ have a disconnected boundary with two components.
Construct a smooth $f$ that is a different constant along each boundary component of each leaf.
Such an $f$ is CR, but does not extend to $H_+$.
\end{example}

\begin{example}
\emph{Higher order terms may complicate the topology of the leaves.}

The submanifold $M$ defined by
\begin{equation}
s=\sin\left(\norm{z}^{-2}\right)e^{-\norm{z}^{-2}}
\end{equation}
is degenerate, and the CR singularity consists of the origin and concentric circles.

Let $f \colon M \to \C$ be $f(z,s) = \norm{z}^2$.  This $f$ is smooth and CR on $M$.
The set $(H_+)_{\{s\}}$ has generally several components and a disconnected boundary.  The function $f$ is equal to a different constant on each component of the boundary of any $(H_+)_{\{s\}}$, and therefore no extension exists to $H_+$.
\end{example}


\def\MR#1{\relax\ifhmode\unskip\spacefactor3000 \space\fi%
  \href{http://www.ams.org/mathscinet-getitem?mr=#1}{MR#1}}



\begin{bibdiv}
\begin{biblist}


\bib{Bishop65}{article}{
   author={Bishop, Errett},
   title={Differentiable manifolds in complex Euclidean space},
   journal={Duke Math.\ J.},
   volume={32},
   date={1965},
   pages={1--21},
   issn={0012-7094},
   review={\MR{0200476}},
}

\bib{Burcea}{article}{
  author={Burcea, Valentin},
  title={A normal form for a real 2-codimensional submanifold in
         $\mathbb{C}^{N+1}$ near a CR singularity},
  journal={Adv.\ Math.},
  volume={243},
  year={2013},
  pages={262--295},
  review={\MR{3062747}},
}

\bib{Burcea2}{article}{
   author={Burcea, Valentin},
   title={On a family of analytic disks attached to a real submanifold
   $M\subset{\mathbb C}^{N+1}$},
   journal={Methods Appl.\ Anal.},
   volume={20},
   date={2013},
   number={1},
   pages={69--78},
   issn={1073-2772},
   review={\MR{3085782}},
}

\bib{Catlin}{article}{
   author={Catlin, David},
   title={Boundary behavior of holomorphic functions on pseudoconvex domains},
   journal={J.\ Differential Geom.},
   volume={15},
   date={1980},
   number={4},
   pages={605--625 (1981)},
   issn={0022-040X},
   review={\MR{0628348}},
}

\bib{Coffman}{article}{
   author={Coffman, Adam},
   title={CR singularities of real fourfolds in $\mathbb C^3$},
   journal={Illinois J. Math.},
   volume={53},
   date={2009},
   number={3},
   pages={939--981 (2010)},
   issn={0019-2082},
   review={\MR{2727363}},
}

\bib{DTZ}{article}{
   author={Dolbeault, Pierre},
   author={Tomassini, Giuseppe},
   author={Zaitsev, Dmitri},
   title={On boundaries of Levi-flat hypersurfaces in ${\mathbb C}^n$},
   language={English, with English and French summaries},
   journal={C.\ R.\ Math.\ Acad.\ Sci.\ Paris},
   volume={341},
   date={2005},
   number={6},
   pages={343--348},
   issn={1631-073X},
   review={\MR{2169149}},
}

\bib{DTZ2}{article}{
   author={Dolbeault, Pierre},
   author={Tomassini, Giuseppe},
   author={Zaitsev, Dmitri},
   title={Boundary problem for Levi flat graphs},
   journal={Indiana Univ.\ Math.\ J.},
   volume={60},
   date={2011},
   number={1},
   pages={161--170},
   issn={0022-2518},
   review={\MR{2952414}},
} 

\bib{HuangFang}{unpublished}{
   author={Fang, Hanlong},
   author={Huang, Xiaojun},
   title={Flattening a non-degenerate CR singular point of real codimension two},
  note={preprint, see \href{http://arxiv.org/abs/1703.09135}{arXiv:1703.09135}}
}

\bib{Gong94:duke}{article}{
   author={Gong, Xianghong},
   title={Normal forms of real surfaces under unimodular transformations
   near elliptic complex tangents},
   journal={Duke Math.\ J.},
   volume={74},
   date={1994},
   number={1},
   pages={145--157},
   issn={0012-7094},
   review={\MR{1271467}},
}

\bib{GongLebl}{article}{
   author={Gong, Xianghong},
   author={Lebl, Ji\v{r}\'\i},
   title={Normal forms for CR singular codimension-two Levi-flat submanifolds},
   journal={Pacific J.\ Math.},
   volume={275},
   date={2015},
   number={1},
   pages={115--165},
}

\bib{Hakim-Sibony}{article}{
   author={Hakim, Monique},
   author={Sibony, Nessim},
   title={Spectre de $A(\overline{\Omega})$ pour des domaines born\'{e}s faiblement pseudoconvexes r\'{e}guliers},
   journal={J.\ Funct.\ Anal.},
   volume={37},
   date={1980},
   number={2},
   pages={127--135},
   issn={0022-1236},
   review={\MR{0578928}},
}

\bib{Harris}{article}{
   author={Harris, Gary Alvin},
   title={The traces of holomorphic functions on real submanifolds},
   journal={Trans.\ Amer.\ Math.\ Soc.},
   volume={242},
   date={1978},
   pages={205--223},
   issn={0002-9947},
   review={\MR{0477120}},
}

\bib{HillTaiani}{article}{
   author={Hill, C. D.},
   author={Taiani, G.},
   title={On the Hans Lewy extension phenomenon in higher codimension},
   journal={Proc. Amer. Math. Soc.},
   volume={91},
   date={1984},
   number={4},
   pages={568--572},
   issn={0002-9939},
   review={\MR{746091}},
}

\bib{HornJohnson}{book}{
   author={Horn, Roger A.},
   author={Johnson, Charles R.},
   title={Matrix analysis},
   edition={2},
   publisher={Cambridge University Press, Cambridge},
   date={2013},
   pages={xviii+643},
   isbn={978-0-521-54823-6},
   review={\MR{2978290}},
}

\bib{HuangKrantz95}{article}{
   author={Huang, Xiaojun},
   author={Krantz, Steven G.},
   title={On a problem of Moser},
   journal={Duke Math.\ J.},
   volume={78},
   date={1995},
   number={1},
   pages={213--228},
   issn={0012-7094},
   review={\MR{1328757}},
}

\bib{HuangYin09}{article}{
   author={Huang, Xiaojun},
   author={Yin, Wanke},
   title={A Bishop surface with a vanishing Bishop invariant},
   journal={Invent.\ Math.},
   volume={176},
   date={2009},
   number={3},
   pages={461--520},
   issn={0020-9910},
   review={\MR{2501295}},
}

\bib{HuangYin09:codim2}{article}{
   author={Huang, Xiaojun},
   author={Yin, Wanke},
   title={A codimension two CR singular submanifold that is formally
   equivalent to a symmetric quadric},
   journal={Int.\ Math.\ Res.\ Not.\ IMRN},
   date={2009},
   number={15},
   pages={2789--2828},
   issn={1073-7928},
   review={\MR{2525841}},
}

\bib{HuangYin:flattening1}{article}{
   author={Huang, Xiaojun},
   author={Yin, Wanke},
   title={Flattening of CR singular points and analyticity of the local hull
   of holomorphy I},
   journal={Math. Ann.},
   volume={365},
   date={2016},
   number={1-2},
   pages={381--399},
   issn={0025-5831},
   review={\MR{3498915}},
}


\bib{HuangYin:flattening2}{article}{
   author={Huang, Xiaojun},
   author={Yin, Wanke},
   title={Flattening of CR singular points and analyticity of the local hull
   of holomorphy II},
   journal={Adv. Math.},
   volume={308},
   date={2017},
   pages={1009--1073},
   issn={0001-8708},
   review={\MR{3600082}},
}

\bib{KenigWebster:82}{article}{
   author={Kenig, Carlos E.},
   author={Webster, Sidney M.},
   title={The local hull of holomorphy of a surface in the space of two
   complex variables},
   journal={Invent.\ Math.},
   volume={67},
   date={1982},
   number={1},
   pages={1--21},
   issn={0020-9910},
   review={\MR{664323}},
}

\bib{LMSSZ}{article}{
   author={Lebl, Ji{\v{r}}{\'{\i}}},
   author={Minor, Andr{\'e}},
   author={Shroff, Ravi},
   author={Son, Duong},
   author={Zhang, Yuan},
   title={CR singular images of generic submanifolds under holomorphic maps},
   journal={Ark.\ Mat.},
   volume={52},
   date={2014},
   number={2},
   pages={301--327},
   issn={0004-2080},
   review={\MR{3255142}},
}

\bib{crext}{article}{
   author={Lebl, Ji{\v{r}}{\'{\i}}},
   author={Noell, Alan},
   author={Ravisankar, Sivaguru},
   title={Extension of CR functions from boundaries in ${\bf C}^n\times{\bf R}$},
   journal={Indiana Univ.\ Math.\ J.},
   volume={},
   date={},
   number={},
   pages={to appear},
   issn={},
   review={},
   note={\href{http://arxiv.org/abs/1505.05255}{arXiv:1505.05255}}
}

\bib{crext2}{article}{
   author={Lebl, Ji{\v{r}}{\'{\i}}},
   author={Noell, Alan},
   author={Ravisankar, Sivaguru},
   title={Codimension two CR singular submanifolds and extensions of CR functions},
   journal = {J.\ Geom.\ Anal.},
   volume={},
   date={2017},
   number={},
   pages={},
   issn={},
   review={},
   note={\href{http://dx.doi.org/10.1007/s12220-017-9767-6}{DOI:10.1007/s12220-017-9767-6},
   \href{http://arxiv.org/abs/1604.02073}{arXiv:1604.02073}}
}

\bib{Malgrange}{book}{
   author={Malgrange, B.},
   title={Ideals of differentiable functions},
   series={Tata Institute of Fundamental Research Studies in Mathematics,
   No. 3},
   publisher={Tata Institute of Fundamental Research, Bombay; Oxford
   University Press, London},
   date={1967},
   pages={vii+106},
   review={\MR{0212575}},
}

\bib{Moser85}{article}{
   author={Moser, J{\"u}rgen K.},
   title={Analytic surfaces in ${\bf C}^2$ and their local hull of
   holomorphy},
   journal={Ann.\ Acad.\ Sci.\ Fenn.\ Ser.\ A I Math.},
   volume={10},
   date={1985},
   pages={397--410},
   issn={0066-1953},
   review={\MR{802502}},
}

\bib{MoserWebster83}{article}{
   author={Moser, J{\"u}rgen K.},
   author={Webster, Sidney M.},
   title={Normal forms for real surfaces in ${\bf C}^{2}$ near complex
   tangents and hyperbolic surface transformations},
   journal={Acta Math.},
   volume={150},
   date={1983},
   number={3--4},
   pages={255--296},
   issn={0001-5962},
   review={\MR{709143}},
}

\bib{Shabat:book}{book}{
   author={Shabat, B.~V.},
   title={Introduction to complex analysis. Part II},
   series={Translations of Mathematical Monographs},
   volume={110},
   note={Functions of several variables;
   Translated from the third (1985) Russian edition by J. S. Joel},
   publisher={American Mathematical Society},
   place={Providence, RI},
   date={1992},
   pages={x+371},
   isbn={0-8218-4611-6},
   review={\MR{1192135}},
}

\bib{Slapar:16}{article}{
   author={Slapar, Marko},
   title={On complex points of codimension 2 submanifolds},
   journal={J.\ Geom.\ Anal.},
   volume={26},
   date={2016},
   number={1},
   pages={206--219},
   issn={1050-6926},
   review={\MR{3441510}},
}

\bib{Tumanov}{article}{
      author={Tumanov, A.~E.},
       title={Extension of {CR}-functions into a wedge from a manifold of finite type},
        date={1988},
        ISSN={0368-8666},
     journal={Mat.\ Sb.\ (N.S.)},
      volume={136(178)},
      number={1},
       pages={128\ndash 139},
      review={\MR{0945904}},
}

\end{biblist}
\end{bibdiv}
\end{document}